\documentclass[a4paper,12pt]{article}

\usepackage{amsfonts,graphicx,subfig,amsmath,amssymb,enumitem,
,url,color,authblk,amsthm,multicol,cite,float}

\DeclareMathOperator\erf{erf}
\DeclareMathOperator\erfc{erfc}

\theoremstyle{definition}

\theoremstyle{theorem}
\newtheorem{theorem}{Theorem}[section]
\theoremstyle{lemma}
\newtheorem{lemma}{Lemma}[section]
\theoremstyle{corollary}
\newtheorem{corollary}{Corollary}[section]
\theoremstyle{remark}
\newtheorem{remark}{Remark}

\title{Similarity solution for a two-phase one-dimensional Stefan problem with a convective boundary condition and a mushy zone model}

\author[1,2]{Andrea N. Ceretani\thanks{aceretani@austral.edu.ar}}
\author[1]{Domingo A. Tarzia\thanks{dtarzia@austral.edu.ar}}
\affil[1]{{\small CONICET - Depto. Matem\'atica, Facultad de Ciencias Empresariales, Universidad Austral, Paraguay 1950, S2000FZF Rosario, Argentina.}}
\affil[2]{{\small Depto. de Matem\'atica, Facultad de Ciencias Exactas, Ingenier\'ia y Agrimensura, Universidad Nacional de Rosario, Pellegrini 250, S2000BTP Rosario, Argentina.}}
\date{}

\begin{document}  
\maketitle

\begin{abstract}
A two-phase solidification process for a one-dimensional semi-infinite material is considered. It is assumed that it is ensued from a constant bulk temperature present in the vicinity of the fixed boundary, which it is modelled through a convective condition (Robin condition). The interface between the two phases is idealized as a mushy region and it is represented following the model of Solomon, Wilson and Alexiades. An exact similarity solution is obtained when a restriction on data is verified, and it is analysed the relation between the problem considered here and the problem with a temperature condition at the fixed boundary. Moreover, it is proved that the solution to the problem with the convective boundary condition converges to the solution to a problem with a temperature condition when the heat transfer coefficient at the fixed boundary goes to infinity, and it is given an estimation of the difference between these two solutions. Results in this article complete and improve the ones obtained in Tarzia, Compt. Appl. Math., 9 (1990), 201-211.
\end{abstract}

\section{Introduction}

Phase-change processes involving solidification or melting are present in a large number of phenomena related to physics, engineering, chemistry, etc. and they have been widely studied since several decades. Some reference books in the subject are \cite{AlSo1993,Ca1984, Cr1984,Fa2005,Gu2003,Lu1991,Ru1971} and a rewiew of a long bibliography on moving and free boundary value problems for the heat equation can be consulted in \cite{Ta2000}. Sometimes,  
%frequently encountered in {\color{blue} completar y agregar bibliograf\'ia}
liquid in solidification processes is cooled until the phase-change temperature without becoming solid. This implies the presence of a region in the phase-change process containing the material at a special solid-liquid state, which is known as {\em mushy region} \cite{AlSo1993,Cr1984,Gu2003}. In this article, we consider a one-dimensional semi-infinite homogeneous material undergoing a two-phase solidification process with a mushy zone. 
This sort of problems were studied in \cite{Ta1990} for boundary conditions of Dirichlet or heat flux type. We follow it, which is inspired by the model given for Solomon, Wilson and Alexiades in \cite{SoWiAl1982} for the one-phase case, to represent the mushy region. Encouraged by the recent relation between the classical (absence of mushy zone) two-phase Stefan problems with temperature and convective boundary conditions \cite{Ta2015-a}, we consider here the following free boundary value problem:
\begin{subequations}\label{PbCondConv}
\begin{align}
\label{EqTheta1Conv}&\alpha_1\theta_{1_{xx}}(x,t)=\theta_{1_{t}}(x,t)&0<x<s(t),&\quad t>0\\
\label{EqTheta2Conv}&\alpha_2\theta_{2_{xx}}(x,t)=\theta_{2_{t}}(x,t)&x>r(t),&\quad t>0\\
\label{s0r0}&s(0)=r(0)=0& &\\
\label{MushyTemp}&\theta_1(s(t),t)=\theta_2(r(t),t)=0& &\quad t>0\\
\label{CondInicEInf}&\theta_2(x,0)=\theta_2(+\infty,t)=\theta_0&x>0,&\quad t>0\\
\label{MushyCalorLat}&k_1\theta_{1_x}(s(t),t)-k_2\theta_{2_x}(r(t),t)=\rho l[\epsilon\dot{s}(t)-(1-\epsilon) \dot{r}(t)]& &\quad t>0\\
\label{MushyAncho}&\theta_{1_x}(s(t),t)(r(t)-s(t))=\gamma& &\quad t>0\\
\label{CondConv}&k_1\theta_{1_x}(0,t)=\frac{h_0}{\sqrt{t}}\left(\theta_1(0,t)+D_\infty\right)& &\quad t>0
\end{align}
\end{subequations}
 
\noindent where the unknowns are:

\vspace{0.25cm}

\begin{tabular}{rll}
$\theta_1:$& temperature of the solid region & [$^\circ$C]\\
$\theta_2:$& temperature of the liquid region & [$^\circ$C]\\
$s:$& free boundary separating the mushy zone and the solid phase& [m]\\
$r:$& free boundary separating the mushy zone and the liquid phase& [m]\\
\end{tabular}

\vspace{0.25cm}

\noindent the physical parameters involved in the model are:

\vspace{0.25cm}

\begin{tabular}{rll}
$\rho>0:$& mass density & [kg/m$^3$]\\
$k>0:$& thermal conductivity & [W/(m$^\circ$C)]\\
$c>0:$& specific heat& [J/(kg$^\circ$C)]\\
$l>0:$& latent heat per unit mass& [J/kg]\\
$0<\epsilon<1:$& coefficient characterizing the amount of latent heat\\& contained in the mushy region& [dimensionless]\\
$\gamma>0:$ & coefficient characterizing the width of the mushy region& [$^\circ$C]\\
$\theta_0>0:$ & initial temperature of the material& [$^\circ$C]\\
$-D_\infty<0:$ & external bulk temperature at the boundary $x=0$& [$^\circ$C]\\
$h_0>0:$& coefficient characterizing the heat transfer at the\\& boundary $x=0$&[kg/($^\circ$C s$^{5/2}$)]\\
$\alpha=\frac{k}{\rho c}>0:$& thermal diffusivity&[$m^2\,s^{-1}$]
\end{tabular}

\vspace{0.25cm}

\noindent and the subscripts 1 and 2 refer to solid and liquid phases, respectively.

We note that we are making the following assumptions on the mushy region \cite{Ta1990,Ta2015-b,SoWiAl1982}:
\begin{enumerate}
\item It is isothermal at the phase-change temperature, which we are considering equal to 0 $^\circ$C.
\item It contains a fixed portion of the total latent heat per unit mass (see condition (\ref{MushyCalorLat})).
\item Its width is inversely proportional to the gradient of temperature (see condition (\ref{MushyAncho})).
\end{enumerate}

We also observe that, by considering the convective boundary condition (\ref{CondConv}), we are thinking of a solidification process ensued due to the constant temperature $-D_\infty$ present in the vicinity of the fixed boundary $x=0$ of the material, which is often represented through physically less appropriate boundary conditions of Dirichlet type \cite{CaJa1959}. Convective boundary conditions have been also used in the context of phase-change processes in, for example, \cite{ZuCh1994, Be1991, CaKw2009,Fo1978,GrHePlSl2013,HuSh1975,Lu2000,RoKa2009,SaSiCoLe2009,
SiGuRa2011,WuWa1994,BrTa1998,Bo1990}. Especially, a heat transfer coefficient inversely proportional to the square root of time it was also considered in \cite{ZuCh1994}.

In the following (Sect. \ref{SectEyU}), we give a characterization for the existence and uniqueness of an explicit similarity solution to problem (\ref{PbCondConv}) in terms of the existence and uniqueness of a positive solution to a transcendental equation. We then prove that it has only one solution if and only if data verify a certain condition. Then (Sect. \ref{RelacTempConv}), we analyse the relation of problem (\ref{PbCondConv}) with the problem (\ref{PbCondConv}$^\star$) given by (\ref{EqTheta1Conv})-(\ref{MushyAncho}) and the following temperature boundary condition:
\begin{equation}\label{CondTempD0}
\theta_1(0,t)=-D_0,\quad t>0\quad\quad(D_0>0)\tag{1h$^\star$}\text{,}
\end{equation}
and we establish when both problems are equivalent. 
%From this, we obtain a bound for the parameter that characterizes the free boundary separating the solid and mushy regions in problem (\ref{PbCondConv}$^\star$). 
Finally (Sect. \ref{SectAssymp}), we prove that the solution to problem (\ref{PbCondConv}) converges to the solution to problem (\ref{PbCondConv}$^\star$)$_\infty$, that is the special case of problem (\ref{PbCondConv}$^\star$) in which the temperature boundary condition is given by:
\begin{equation*}\label{CondTempDInfty}
\hspace*{6cm}\theta_1(0,t)=-D_\infty,\quad t>0,\hspace{5cm}(1\text{h}^\star)_\infty
\end{equation*}
when the heat transfer coefficient goes to infinity. Moreover, we obtain that the difference between the two solutions is $\mathcal{O}\left(\frac{1}{h_0}\right)$ when $h_0\to\infty$.

\section{Existence and uniqueness of solution}\label{SectEyU}

In this section we will look for a similarity solution to problem (\ref{PbCondConv}). By following the classical method of Neumann \cite{RiWe1912}, that is, by introducing the similarity variables:
\begin{equation*}
\eta_1=\frac{x}{2\sqrt{\alpha_1}t}\quad\text{and}\quad
\eta_2=\frac{x}{2\sqrt{\alpha_2}t}
\end{equation*}
and proposing a solution defined by:
\begin{align*}
&\theta_1(x,t)=\theta_1(\eta_1)&0<x<s(t),&\quad t>0\\
&\theta_2(x,t)=\theta_2(\eta_2)&x>r(t),&\quad t>0\\
&s(t)=2\xi\sqrt{\alpha_1t}& &\quad t>0\\
&r(t)=2\mu\sqrt{\alpha_2t}& &\quad t>0
\end{align*}
with $\xi$ and $\mu$ positive numbers to be determined, we obtain that $\theta_1$ and $\theta_2$ must be given by:
\begin{align*}
&\theta_1(x,t)=A_1+B_1\erf(\eta_1)&0<x<s(t),&\quad t>0\\
&\theta_2(x,t)=A_2+B_2\erf(\eta_2)&x>r(t),&\quad t>0
\end{align*}
where $A_1$, $A_2$, $B_1$, $B_2$ are real numbers that must be specified from conditions (\ref{MushyTemp})-(\ref{CondConv}), and $\erf$ is the {\em error function} defined by:
\begin{equation*}
\erf(x)=\frac{2}{\sqrt{\pi}}\displaystyle\int_0^x\exp(-y^2)\,dy,\quad x>0.
\end{equation*}

\noindent Through conditions (\ref{MushyTemp}), (\ref{CondConv}) we obtain that:
\begin{equation*}
A_1=-\frac{D_\infty\erf(\xi)}{\erf(\xi)+\frac{k_1}{h_0\sqrt{\alpha_1\pi}}}
\quad\text{and}\quad
B_1=\frac{D_\infty}{\erf(\xi)+\frac{k_1}{h_0\sqrt{\pi\alpha_1}}}\text{,}
\end{equation*} 
and from conditions (\ref{MushyTemp}), (\ref{CondInicEInf}) that:
\begin{equation*}
A_2=-\frac{\theta_0\erf(\mu)}{\erfc(\mu)}
\quad\text{and}\quad
B_2=\frac{\theta_0}{1-\erfc(\mu)}\text{,}
\end{equation*}
where $\erfc$ is the {\em complementary error function} defined by:
\begin{equation*}
\erfc(x)=1-\erf(x),\quad x>0.
\end{equation*} 
Exploiting condition (\ref{MushyAncho}) we have that the parameters $\xi$ and $\mu$, which characterize the two free boundaries of the mushy region, are related as:
\begin{equation}\label{mu}
\mu=\sqrt{\alpha_{12}}W(\xi)\text{,}
\end{equation}
where $\alpha_{12}$ is the number defined by:
\begin{equation*}
\alpha_{12}=\frac{\alpha_1}{\alpha_2}>0
\end{equation*} 
and $W$ is the function defined by:
\begin{equation}\label{W}
W(x)=x+\frac{\gamma\sqrt{\pi}}{2D_\infty}\exp(x^2)\left(\erf(x)+\frac{k_1}{h_0\sqrt{\alpha_1\pi}}\right),\quad x>0.
\end{equation}

\noindent Finally, through condition (\ref{MushyCalorLat}), we have that $\xi$ must be such that:
\begin{equation*}
F(\xi)=\frac{l\sqrt{\pi}}{D_\infty c_1}G(\xi)\text{,}
\end{equation*}
where $F$ and $G$ are the functions defined by:
\begin{subequations}\label{FG}
\begin{align}
\label{F}&F(x)=\frac{\exp(-x^2)}{\erf(x)+\frac{k_1}{h_0\sqrt{\alpha_1\pi}}}-
\frac{\theta_0\sqrt{k_2c_2}}{D_\infty\sqrt{k_1c_1}}\frac{\exp\left(-\alpha_{12}W^2(x)\right)}{\erfc\left(\sqrt{\alpha_{12}}W(x)\right)}&x>0\\
\label{G}&G(x)=x+\frac{(1-\epsilon)\gamma\sqrt{\pi}}{2D_\infty}\exp(x^2)\left(\erf(x)+\frac{k_1}{h_0\sqrt{\alpha_1\pi}}\right)&x>0.
\end{align}
\end{subequations}

Then, we have the following result:
\begin{theorem}\label{ThCharact}
The Stefan problem (\ref{PbCondConv}) has the similarity solution $\theta_1$, $\theta_2$, $s$, $r$ given by:
\begin{subequations}\label{SolConv}
\begin{align}
&\theta_1(x,t)=-\frac{D_\infty\erf(\xi)}{\erf(\xi)+\frac{k_1}{h_0\sqrt{\pi\alpha_1}}}\left(1-\frac{\erf\left(\frac{x}{2\sqrt{\alpha_1 t}}\right)}{\erf(\xi)}\right)&0<x<s(t),&\quad t>0\\
&\theta_2(x,t)=\frac{\theta_0\erf(\mu)}{\erfc(\mu)}\left(\frac{\erf\left(\frac{x}{2\sqrt{\alpha_2 t}}\right)}{\erf(\mu)}-1\right)&x>r(t),&\quad t>0\\
&s(t)=2\xi\sqrt{\alpha_1 t} &&\quad t>0\\
&r(t)=2\mu\sqrt{\alpha_2 t} &&\quad t>0
%&r(t)=2W_2(\xi)\sqrt{\alpha_1t} &&\quad t>0
\end{align}
\end{subequations}
with $\mu$ given by (\ref{mu}), if and only if $\xi$ is a solution to the equation:
\begin{equation}\label{EqXi}
F(x)=\frac{l\sqrt{\pi}}{D_\infty c_1}G(x)\text{,}\quad x>0,
\end{equation}
where $F$ and $G$ are the functions defined in (\ref{FG}).
\end{theorem}

Therefore, finding a similarity solution to problem (\ref{PbCondConv}) reduces to studying equation (\ref{EqXi}). We begin this by introducing some functions related to equation (\ref{EqXi}) and some properties of them. Let be $F_1$, $F_2$ the functions defined by:
\begin{align}
\label{F1}&F_1(x)=\frac{\exp(-x^2)}{\erf(x)+\frac{k_1}{h_0\sqrt{\alpha_1\pi}}},&x>0\\
\label{F2}&F_2(x)=\frac{\exp(-x^2)}{\erfc(x)},&x>0
\end{align}
Then, (\ref{F}) can be rewritten as:
\begin{equation}\label{FBis}
F(x)=F_1(x)-\frac{\theta_0\sqrt{k_2c_2}}{D_\infty\sqrt{k_1c_1}}F_2\left(\sqrt{\alpha_{12}}W(x)\right),\quad x>0.
\end{equation}

\begin{lemma}\label{LeFGProp}
$$\vspace*{-1cm}$$
\noindent 1. The functions $W$, $F_1$, $F_2$ defined by (\ref{W}), (\ref{F1}), (\ref{F2}), respectively, verify:
\begin{subequations}
\begin{align}
\label{WProp}&W(0^+)=\frac{\gamma k_1}{2D_\infty h_0\sqrt{\alpha_1}}, &
W&(+\infty)=+\infty, &
W'(x)>0\quad\forall\,x>0\\
\label{F1Prop}&F_1(0^+)=\frac{h_0\sqrt{\alpha_1\pi}}{k_1}>0, &
F_1&(+\infty)=0, &
F_1'(x)<0\quad\forall\,x>0\\
\label{F2Prop}&F_2(0^+)=1, &
F_2&(+\infty)=+\infty, &
F'_2(x)>0\quad\forall\,x>0
\end{align}
\end{subequations}

\noindent 2. The function $F$ defined by (\ref{F}) verifies:
\begin{equation}\label{FProp}
{\small
\hspace*{-0.25cm}F(0^+)=\frac{h_0\sqrt{\alpha_1\pi}}{k_1}-\frac{\theta_0\sqrt{k_2c_2}}{D_\infty\sqrt{k_1c_1}}F_2\left(\frac{\gamma k_1}{2D_\infty h_0\sqrt{\alpha_2}}\right),
\hspace*{0.25cm}
F(+\infty)=-\infty,
\hspace*{0.25cm}
F'(x)<0\quad\forall\,x>0
}
\end{equation}
\noindent 3. The function $G$ defined by (\ref{G}) verifies:
\begin{equation}\label{GProp}
G(0^+)=\frac{(1-\epsilon)\gamma k_1}{2D_\infty h_0\sqrt{\alpha_1}},
\hspace*{0.5cm}
G(+\infty)=+\infty, 
\hspace*{0.5cm}
G'(x)>0\quad\forall\,x>0.
\end{equation}
\end{lemma}

\begin{proof}
It follows from elementary computations.
\end{proof}  

Then, we have:

\begin{theorem}\label{ThEyUEqXi}
Equation (\ref{EqXi}) has an only one positive solution if and only if the coefficient $h_0$ verifies the following inequality:
\begin{equation}\label{IneqH0}
h_0>h_0^\star,
\end{equation}
where $h_0^\star$ is defined by:
\begin{equation}\label{h0Star}
h_0^\star=\frac{\gamma k_1}{2D_\infty\eta\sqrt{\alpha_2}},
\end{equation}
with $\eta=\eta\left(\frac{\gamma k_1}{\theta_0 k_2},\frac{(1-\epsilon)l}{\theta_0 c_2}\right)$ the only one solution to the equation:
\begin{equation}\label{EqEta0}
F_3(x)=0,\quad x>0,
\end{equation}
and $F_3$ the function defined by:
\begin{equation}\label{F3}
F_3(x)=F_2(x)-\frac{\gamma k_1\sqrt{\pi}}{2\theta_0k_2}\frac{1}{x}+\frac{(1-\epsilon)l\sqrt{\pi}}{\theta_0c_2}x,\quad x>0.
\end{equation}
\end{theorem}

\begin{proof}
It follows from the properties of the functions $F$, $G$ given in Lemma \ref{LeFGProp} that equation (\ref{EqXi}) admits an only one positive solution if and only if:
\begin{equation}\label{F0+G0+}
F(0^+)>\frac{l\sqrt{\pi}}{D_\infty c_1}G(0^+).
\end{equation}
Let us observe that, by using the function $F_3$ given by (\ref{F3}), (\ref{F0+G0+}) can be rewritten as:
\begin{equation}\label{F0+G0+Bis}
F_3\left(\frac{\gamma k_1}{2D_\infty h_0\sqrt{\alpha_2}}\right)<0.
\end{equation}
Let be $F_4$ the function defined by:
\begin{equation*}
F_4(x)=\frac{\gamma k_1\sqrt{\pi}}{2\theta_0k_2}\frac{1}{x}-\frac{(1-\epsilon)l\sqrt{\pi}}{\theta_0c_2}x,\quad x>0.
\end{equation*}
Since:
\begin{equation*}
F_3(x)=F_2(x)-F_4(x),\quad x>0,
\end{equation*}
it follows from the properties of the function $F_2$ given in Lemma \ref{LeFGProp} and the fact that $F_4$ verifies:
\begin{equation*}
F_4(0^+)=-\infty,
\hspace{1cm}
F_4(+\infty)=+\infty,
\hspace{1cm}
F_4'(x)<0\quad x>0,
\end{equation*}
that $F_3$ is such that:
\begin{equation*}
F_3(0^+)=-\infty,
\hspace{1cm}
F_3(+\infty)=+\infty,
\hspace{1cm}
F_3'(x)>0\quad x>0.
\end{equation*}
Therefore, (\ref{F0+G0+Bis}) holds if and only if:
\begin{equation}\label{F0+G0+BisBis}
0<\frac{\gamma k_1}{2D_\infty h_0\sqrt{\alpha_2}}<\eta,
\end{equation}
where $\eta=\eta\left(\frac{\gamma k_1}{\theta_0 k_2},\frac{(1-\epsilon)l}{\theta_0 c_2}\right)$ is the only one positive solution to equation (\ref{EqEta0}). Only remains to observe that inequality (\ref{F0+G0+BisBis}) is equivalent to (\ref{IneqH0}).
\end{proof}

From Theorems \ref{ThCharact} and \ref{ThEyUEqXi}, we can establish now the main result of this section:

\begin{corollary}\label{ThEyU}
The Stefan problem (\ref{PbCondConv}) has the similarity solution given by (\ref{SolConv}) if and only if the coefficient $h_0$ that characterizes the heat transfer coefficient at the boundary $x=0$ is large enough so much as to verifies inequality (\ref{IneqH0}). 
\end{corollary}

\begin{remark}
%{\color{blue} Coherencia con el problema a una fase estudiado en [Ta2015-a]. Hoja C12.}\\
In \cite{Ta2015-b} it was obtained an explicit similarity solution for a one-phase solidification process with a mushy zone according to the model of Solomon, Wilson and Alexiades \cite{SoWiAl1982}. We note that Theorem \ref{ThEyU} reduces to Theorem 1 in \cite{Ta2015-b}, in which the explicit solution is established, when it is considered an initial temperature for the liquid phase equal to the phase-change temperature. That is, when $\theta_0=0$ we have that the solution presented in this article coincides with the solution given in \cite{Ta2015-b} and that the hypothesis on the heat transfer coefficient under which we have the solution is equivalent to the one given in \cite{Ta2015-b}.

%{\color{blue} Coherencia con el problema a dos fases sin zona pastosa estudiado en [Ta2004]. Hojas C14 y C15, más mis apuntes.}
In \cite{Ta2015-a} it was obtained an explicit similarity solution for a two-phase solidification process without any mushy region. We also have that Theorem \ref{ThEyU} reduces to Theorem 2 in \cite{Ta2015-a}, in which the explicit solution is obtained, if we think of a mushy region of zero thickness. In other words, when $\gamma=0$ we have that the solution obtained here coincides with the solution given in \cite{Ta2015-a} and that the condition for the heat transfer coefficient is equivalent to the one given there. 
\end{remark}

\section{Relation between the problems with convective and temperature boundary conditions}\label{RelacTempConv}
As we have mentioned before, convective boundary conditions are physically more appropriate to represent a temperature imposed at the boundary of a material (actually, in the vicinity of) than conditions of Dirichlet type \cite{CaJa1959}. Nevertheless, Dirichlet conditions are frequently encountered in the literature modelling this sort of situations. Thus we are interested in analysing the relationship between the problems with the two types of conditions. In other words, in how problems (\ref{PbCondConv}) and (\ref{PbCondConv}$^\star$) are related. 

Let us start by considering problem (\ref{PbCondConv}) with $h_0$ satisfying condition (\ref{IneqH0}). We know from Corollary \ref{ThEyU} that it has the similarity solution given by (\ref{SolConv}), where $\xi$ is the only one positive solution to equation (\ref{EqXi}). Since:
\begin{equation*}
\theta_1(0,t)=-\frac{D_\infty\erf(\xi)}{\erf(\xi)+\frac{k_1}{h_0\sqrt{\pi\alpha_1}}}\text{,}
\end{equation*}
we will consider problem (\ref{PbCondConv}$^\star$) with $D_0$ defined as:
\begin{equation}\label{D0}
D_0=\frac{D_\infty\erf(\xi)}{\erf(\xi)+\frac{k_1}{h_0\sqrt{\pi\alpha_1}}}>0\text{.}
\end{equation}
We know from \cite{Ta1990} that this problem has the similarity solution given by:
\begin{subequations}\label{SolTemp}
\begin{align}
&\theta^\star_1(x,t)=-D_0\left(1-\frac{\erf\left(\frac{x}{2\sqrt{\alpha_1 t}}\right)}{\erf(\xi^\star)}\right)
&0<x<s^\star(t),&\quad t>0\\
&\theta^\star_2(x,t)=\frac{\theta_0\erf(\mu^\star)}{\erfc(\mu^\star)}\left(\frac{\erf\left(\frac{x}{2\sqrt{\alpha_2t}}\right)}{\erf(\mu^\star)}-1\right)
&x>r^\star(t),&\quad t>0\\
&s^\star(t)=2\xi^\star\sqrt{\alpha_1 t} &&\quad t>0\\
&r^\star(t)=2\mu^*\sqrt{\alpha_2t} &&\quad t>0
\end{align}
\end{subequations}
where $\mu^*$ is given by:
\begin{equation}\label{muStar}
\mu^*=\sqrt{\alpha_{12}}W_0(\xi^\star),
\end{equation}
$\xi^\star$ is the only one solution to the equation:
\begin{equation}\label{EqXiStar}
F_0(x)=\frac{l\sqrt{\pi}}{D_0 c_1}G_0(x)\text{,}\quad x>0
\end{equation}
and $W_0$, $F_0$, $G_0$ are the functions defined by:
\begin{subequations}\label{W0F0G0}
\begin{align}
\label{W0}&W_0(x)=x+\frac{\gamma\sqrt{\pi}}{2D_0}\exp(x^2)\erf(x)&x>0\\
\label{F0}&F_0(x)=\frac{\exp(-x^2)}{\erf(x)}-\frac{\theta_0\sqrt{k_2c_2}}{D_0\sqrt{k_1c_1}}\frac{\exp\left(-\alpha_{12}{W_0}^2(x)\right)}{\erfc\left(\sqrt{\alpha_{12}}W_0(x)\right)}&x>0\\
\label{G0}&G_0(x)=x+\frac{(1-\epsilon)\gamma\sqrt{\pi}}{2D_0}\exp(x^2)\erf(x)&x>0.
\end{align}
\end{subequations}
Exploiting the fact that $\xi$ satisfies (\ref{EqXi}), it follows that it is also a solution to equation (\ref{EqXiStar}). In fact, when $D_0$ is given by (\ref{D0}), we have that: 
\begin{equation*}
\begin{split}
F_0(\xi)&=
\frac{\exp(-\xi^2)}{\erf(\xi)}-
\frac{\theta_0\sqrt{k_2c_2}}{D_\infty\sqrt{k_1c_1}}\frac{\erf(\xi)+\frac{k_1}{h_0\sqrt{\pi\alpha_1}}}{\erf(\xi)}
F_2\left(\sqrt{\alpha_{12}}\left(\xi+\frac{\gamma\sqrt{\pi}}{2D_\infty}\exp(\xi^2)\left(\erf(\xi)+\frac{k_1}{h_0\sqrt{\pi\alpha_1}}\right)\right)\right)\\
&=\frac{\erf(\xi)+\frac{k_1}{h_0\sqrt{\pi\alpha_1}}}{\erf(\xi)}\left[F_1(\xi)-
\frac{\theta_0\sqrt{k_2c_2}}{D_\infty\sqrt{k_1c_1}}
F_2\left(\sqrt{\alpha_{12}}W(\xi)\right)\right]\\
&=\frac{\erf(\xi)+\frac{k_1}{h_0\sqrt{\pi\alpha_1}}}{\erf(\xi)}F(\xi)=
\frac{\erf(\xi)+\frac{k_1}{h_0\sqrt{\pi\alpha_1}}}{\erf(\xi)}\left[\frac{l\sqrt{\pi}}{D_\infty c_1}G(\xi)\right]
\\
&=\frac{\erf(\xi)+\frac{k_1}{h_0\sqrt{\pi\alpha_1}}}{\erf(\xi)}\left[\frac{l\sqrt{\pi}}{D_0 c_1}\frac{\erf(\xi)}{\erf(\xi)+\frac{k_1}{h_0\sqrt{\pi\alpha_1}}}\left(\xi+\frac{(1-\epsilon)\gamma\sqrt{\pi}}{2D_0}\exp(\xi^2)\erf(\xi)\right)\right]\\
&=\frac{l\sqrt{\pi}}{D_0 c_1}G_0(\xi)
\text{.}
\end{split}
\end{equation*}
Therefore, $\xi=\xi^\star$. From this, it is easy to see that $\mu=\mu^\star$, $\theta_1=\theta_1^\star$ and $\theta_2=\theta_2^\star$. 

Then, we have the following theorem:

\begin{theorem}\label{ThConvToTemp}
If $h_0$ satisfies condition (\ref{IneqH0}) then the similarity solution (\ref{SolConv}) to problem (\ref{PbCondConv}) coincides with the similarity solution (\ref{SolTemp}) to problem (\ref{PbCondConv}$^\star$) when $D_0$ is given by (\ref{D0}).
\end{theorem}

Let us consider now the problem (\ref{PbCondConv}$^\star$). It follows from \cite{Ta1990} that it has the similarity solution given by (\ref{SolTemp}), where $\xi^\star$ is the only one positive solution to equation (\ref{EqXiStar}). Let $D_\infty>D_0$ and let $h_0>0$. Since:
\begin{equation*}
k_1\theta_1^\star(0,t)=\frac{h_0}{\sqrt{t}}(\theta_1^\star(0,t)+D_\infty)
\end{equation*}
if and only if:
\begin{equation}\label{h0}
h_0=\frac{k_1D_0}{\sqrt{\pi\alpha_1}(D_\infty-D_0)\erf(\xi^\star)}>0\text{,}
\end{equation}
we will consider problem (\ref{PbCondConv}) with $D_\infty>D_0$ and $h_0$ given by (\ref{h0}). As before, by taking into account that $\xi^\star$ satisfies equation (\ref{EqXiStar}), it can be shown that $\xi^\star$ is a solution to equation (\ref{EqXi}). Then, we have from Theorem \ref{ThCharact} that problem (\ref{PbCondConv}) admits the similarity solution given by (\ref{SolConv}) with $\xi=\xi^\star$. Moreover, Corollary \ref{ThCharact} implies that $h_0$ satisfies (\ref{IneqH0}), which in this case can be written as:
\begin{equation}\label{IneqXiStarPreliminar}
\erf(\xi^\star)<\frac{2D_\infty D_0\eta}{\gamma(D_\infty-D_0)\sqrt{\pi\alpha_{12}}}\text{.}
\end{equation}
Then, we have the following theorem:
\begin{theorem}\label{ThTempToConv}
The similarity solution (\ref{SolTemp}) to problem (\ref{PbCondConv}$^\star$) coincides with the similarity solution (\ref{SolConv}) to problem (\ref{PbCondConv}) when $D_\infty>D_0$ and $h_0$ is given by (\ref{h0}). Moreover, the parameter $\xi^\star$ that characterizes the free boundary separating the solid phase and the mushy region verifies the following inequality:
\begin{equation}
\erf(\xi^\star)<\min\left\{1,\frac{2D_\infty D_0\eta}{\gamma(D_\infty-D_0)\sqrt{\pi\alpha_{12}}}\right\}\text{,}
\end{equation}
where $\eta$ is the only one solution to equation (\ref{EqEta0}). 
\end{theorem}

Therefore, in the sense established by Theorems \ref{ThConvToTemp} and \ref{ThTempToConv}, we have that problems (\ref{PbCondConv}) and (\ref{PbCondConv}$^\star$) are equivalent.

\begin{corollary}
The parameter $\xi^\star$ that characterizes the free boundary separating the solid and mushy regions in problem (\ref{PbCondConv}$^\star$) verifies the following inequality:
\begin{equation}\label{IneqXiStar}
\erf(\xi^\star)\leq\min\left\{1,\frac{2D_0\eta}{\gamma\sqrt{\pi\alpha_{12}}}\right\},
\end{equation}
where $\eta$ is the only one solution to equation (\ref{EqEta0}). 
\end{corollary}

\begin{proof}
It follows by making $D_\infty\to\infty$ into both sides of (\ref{IneqXiStarPreliminar}).
\end{proof}

\begin{remark}
Inequality (\ref{IneqXiStar}), which is physically relevant when $\frac{2D_0\eta}{\gamma\sqrt{\pi\alpha_{12}}}<1$, has already been obtained in \cite{Ta1990} through the relationship between problem (\ref{PbCondConv}$^\star$) and the problem consisting in (\ref{EqTheta1Conv}) to (\ref{MushyAncho}) and the following flux boundary condition:
\begin{equation*}
k_1\theta_{1x}(0,t)=\frac{q_0}{\sqrt{t}},\quad t>0\hspace{2cm}(q_0>0).
\end{equation*}
\end{remark}

\section{Assymptotic behaviour when $h_0\to+\infty$}\label{SectAssymp}
From a physical point of view, if we were able to consider an infinite heat transfer coefficient at $x=0$, the convective boundary condition (\ref{CondConv}) could be replaced by the temperature boundary condition $(1\text{h}^\star)_\infty$. Thus, it is reasonable to expect that the solution to problem (\ref{PbCondConv}) converges to the solution to problem (\ref{PbCondConv}$^\star$)$_\infty$ when the heat transfer coefficient increases its value. In this section we will analyse this sort of convergence, which was already proved for some other Stefan problems in \cite{CeTa2014,CeTa2015,CeTa2016}.

For each $h_0$ satisfying (\ref{IneqH0}) we will consider problem (\ref{PbCondConv}) and we will denote its solution as $\theta_{1,h_0}$, $\theta_{2,h_0}$, $s_{h_0}$, $r_{h_0}$. The solution to problem (\ref{PbCondConv}$^\star$)$_\infty$ will be referred to as $\theta_{1,\infty}^\star$, $\theta_{2,\infty}^\star$, $s_\infty^\star$, $r_\infty^\star$.

The main result of this section is the following:

\begin{theorem}\label{ThOrden}
The solution to problem (\ref{PbCondConv}) given by (\ref{SolConv}) punctually converges to the solution to problem (\ref{PbCondConv}$^\star$)$_\infty$ given by (\ref{SolTemp}), when $h_0\to\infty$. Moreover, the following estimations holds when $h_0\to\infty$:
\begin{subequations}\label{SolOrden}
\begin{align}
\label{theta1Orden}&\theta_{1,h_0}(x,t)-\theta_{1,\infty}(x,t)=\mathcal{O}\left(\frac{1}{h_0}\right)&\forall\,x>0,\,t>0\\
\label{theta2Orden}&\theta_{2,h_0}(x,t)-\theta_{2,\infty}(x,t)=\mathcal{O}\left(\frac{1}{h_0}\right)&\forall\,x>0,\,t>0\\
\label{sOrden}&s_{h_0}(t)-s_\infty(t)=\mathcal{O}\left(\frac{1}{h_0}\right)&t>0\\
\label{rOrden}&r_{h_0}(t)-r_\infty(t)=\mathcal{O}\left(\frac{1}{h_0}\right)&t>0.
\end{align}
\end{subequations}
\end{theorem}

The key to prove Theorem \ref{ThOrden} is the fact that $\xi_{h_0}-\xi_\infty=\mathcal{O}\left(\frac{1}{h_0}\right)$ when $h_0\to\infty$. We will first prove it and then we will back and give the demonstration of Theorem \ref{ThOrden}.

Hereinafter, we will refer to the functions $F$, $G$, $W$, $F_1$ related to problem (\ref{PbCondConv}), as $F_{h_0}$, $G_{h_0}$, $W_{h_0}$, $F_{1,h_0}$, respectively. Analogously, we will refer to the functions $F_0$, $G_0$, $W_0$ associated with condition $(1\text{h}^\star)_\infty$, as $F_\infty$, $G_\infty$, $W_\infty$. That is, $F_\infty$, $G_\infty$, $W_\infty$ will be the functions defined by:
\begin{subequations}\label{FInfGInfWInf}
\begin{align}
\label{FInf}&F_\infty(x)=\frac{\exp(-x^2)}{\erf(x)}-\frac{\theta_0\sqrt{k_2c_2}}{D_\infty\sqrt{k_1c_1}}\frac{\exp\left(-\alpha_{12}W_\infty^2(x)\right)}{\erfc\left(\sqrt{\alpha_{12}}W_\infty(x)\right)}&x>0\\
\label{GInf}&G_\infty(x)=x+\frac{(1-\epsilon)\gamma\sqrt{\pi}}{2D_\infty}\exp(x^2)\erf(x)&x>0\\
\label{WInf}&W_\infty(x)=x+\frac{\gamma\sqrt{\pi}}{2D_\infty}\exp(x^2)\erf(x)&x>0.
\end{align}
\end{subequations}
Finally, let be $J_{h_0}$ $J_\infty$ the functions defined by:
\begin{subequations}\label{JJInf}
\begin{align}
\label{J}&J_{h_0}(x)=\frac{F_{h_0}(x)}{G_{h_0}(x)},&x>0\\
\label{JInf}&J_\infty(x)=\frac{F_\infty(x)}{G_\infty(x)},&x>0.
\end{align}
\end{subequations}
By using the functions $H_{h_0}$, $H_\infty$ defined by:
\begin{subequations}\label{HHInf}
\begin{align}
\label{H}&H_{h_0}(x)=\frac{G_{h_0}(x)}{F_{1,h_0}(x)},&x>0\\
\label{HInf}&H_\infty(x)=\frac{G_\infty(x)}{F_{1,\infty(x)}},&x>0,
\end{align}
\end{subequations}
where the $F_{1,\infty}$ is the function given by:
\begin{equation*}
F_{1,\infty}(x)=\frac{\exp(-x^2)}{\erf(x)},\quad x>0,
\end{equation*}
it follows that (\ref{JJInf}) can be written as:
\begin{subequations}\label{JJInfBis}
\begin{align}
\label{JBis}&J_{h_0}(x)=\frac{1}{H_{h_0}(x)}-\frac{\theta_0\sqrt{k_2c_2}}{D_\infty\sqrt{k_1c_1}}\frac{F_2\left(\sqrt{\alpha_{12}}W_{h_0}(x)\right)}{G_{h_0}(x)},&x>0\\
\label{JInfBis}&J_{\infty}(x)=\frac{1}{H_{\infty}(x)}-\frac{\theta_0\sqrt{k_2c_2}}{D_\infty\sqrt{k_1c_1}}\frac{F_2\left(\sqrt{\alpha_{12}}W_{\infty}(x)\right)}{G_{\infty}(x)},&x>0
\end{align}
\end{subequations}

\begin{lemma}\label{LeJProp}
$$\vspace*{-1cm}$$
\begin{enumerate}
\item The function $J_{h_0}$ defined by (\ref{J}) verifies:
\begin{subequations}
\begin{align}
\label{J0+}&J_{h_0}(0^+)>0&\forall\,h_0\geq h_1^\star\\
\label{JDer}&J'_{h_0}(x)<0&\forall\,x\in(0,\nu_{h_0}),\,\forall\,h_0\geq h_1^\star,
\end{align}
\end{subequations}
where $h_1^\star$ is a positive number such that:
\begin{equation}\label{h1Star}
\frac{1}{h_1^\star}F_2\left(\frac{\gamma k_1}{2D_\infty\sqrt{\alpha_2}}\frac{1}{h_1^\star}\right)<\zeta,
\end{equation}
with:
\begin{equation}\label{zeta}
\zeta=\frac{D_\infty\sqrt{\pi}}{\theta_0\sqrt{\rho k_2c_2}},
\end{equation}
and $\nu_{h_0}$ is the only one solution to the equation:
\begin{equation}\label{EqNu}
J_{h_0}(x)=0,\quad x>0,\,h_0\geq h_1^\star.
\end{equation}
\item The function $J_\infty$ defined by (\ref{JInf}) verifies:
\begin{subequations}
\begin{align}
\label{JInf0+}&J_\infty(0^+)=+\infty\\
\label{JInfDer}&J'_\infty(x)<0,\quad\forall\,x\in(0,\nu_\infty),
\end{align}
\end{subequations}
where $\nu_\infty$ is the only one solution to the equation:
\begin{equation}\label{EqNuInf}
J_\infty(x)=0,\quad x>0.
\end{equation}
\end{enumerate}
\end{lemma}

\begin{proof}
$$\vspace*{-1cm}$$
\begin{enumerate}
\item We have from Lemma \ref{LeFGProp} that:
\begin{subequations}\label{Term0+}
\begin{align*}
&\frac{1}{H_{h_0}(0^+)}=\frac{2D_\infty\alpha_1\sqrt{\pi}}{(1-\epsilon)\gamma}\left(\frac{h_0}{k_1}\right)^2\\
&\frac{F_2\left(\sqrt{\frac{\alpha_1}{\alpha_2}}W_{h_0}(0^+)\right)}{G_{h_0}(0^+)}=\frac{2D_\infty h_0\sqrt{\alpha_1}}{(1-\epsilon)\gamma k_1}F_2\left(\frac{\gamma k_1}{2D_\infty h_0\sqrt{\alpha_2}}\right).
\end{align*}
\end{subequations}
Then:
\begin{equation}\label{J0+Bis}
J_{h_0}(0^+)=\frac{2D_\infty\alpha_1\sqrt{\pi}}{(1-\epsilon)\gamma}\left(\frac{h_0}{k_1}\right)^2\left(1-\frac{1}{h_0\zeta}F_2\left(\frac{\gamma k_1}{2D_\infty h_0\sqrt{\alpha_2}}\right)\right),
\end{equation}
where $\zeta$ is defined by (\ref{zeta}).
Therefore, $J_{h_0}(0^+)>0$ if and only if:
\begin{equation}\label{J0+Pos}
\frac{1}{h_0}F_2\left(\frac{\gamma k_1}{2D_\infty\sqrt{\alpha_2}}\frac{1}{h_0}\right)<\zeta.
\end{equation}
Let be $F_5$ the function defined by:
\begin{equation*}
F_5(x)=\frac{1}{x}F_2\left(\frac{1}{x}\right),\quad x>0.
\end{equation*}
Since $F_5$ verifies:
\begin{equation*}
F_5(0^+)=+\infty,
\hspace{1cm}
F_5(+\infty)=0,
\hspace{1cm}
F_5'(x)<0\quad\forall\,x>0,
\end{equation*}
it follows that there exists a positive number $h_1^\star\geq h_0^\star$ which verifies (\ref{h1Star}). Moreover, as we know from Lemma \ref{LeFGProp} that $F_2$ is an increasing function, we have that (\ref{J0+Pos}) holds for any $h_0\geq h_1^\star$.

It follows from (\ref{J0+}) and the properties of the function $F_{h_0}$ given in Lemma \ref{LeFGProp}, that there exists an only one solution $\nu_{h_0}$ to the equation (\ref{EqNu}) for any $h_0\geq h_1^\star$. Moreover, since:
\begin{equation}\label{FPos}
F_{h_0}(x)>0\quad\forall\,x\in(0,\nu_{h_0}),\,\forall\,h_0\geq h_1^\star,
\end{equation} 
it follows from the Leibnitz rule and the properties of the functions $F'_{h_0}$, $G'_{h_0}$ given in Lemma \ref{LeFGProp} that (\ref{JDer}) holds.
\item It is similar to the proof given for $J_{h_0}$ in the previous item.
\end{enumerate}
\end{proof}

\begin{lemma}\label{LeJConv}
$$\vspace*{-1cm}$$
\begin{enumerate}
\item Let be $h_1^\star$ as in Lemma \ref{LeJProp}. The sequence of functions $\left\{J_{h_0}\right\}_{h_0\geq h_1^\star}$ has the following properties:
\begin{enumerate}
\item $J_{h_0}(x)\to J_\infty(x)$ when $h_0\to\infty$, for all $x\in\mathbb{R}^+$.
\item If $h_1^\star\leq h_0^{(1)}<h_0^{(2)}$, then:
\begin{equation}\label{JMonot}
J_{h_0^{(1)}}(x)<J_{h_0^{(2)}}(x)\quad\forall\,x\in(0,\nu_{h_0^{(1)}}),
\end{equation}
where $\nu_{h_0^{(1)}}$ is defined as in Lemma \ref{LeJProp}.
\end{enumerate}  
\item $\left\{\xi_{h_0}\right\}_{h_0\geq h_1^\star}$ is an increasing sequence of numbers which converges to $\xi_\infty$ when $h_0\to\infty$.
\end{enumerate}
\end{lemma}

\begin{proof}
$$\vspace*{-1cm}$$
\begin{enumerate}
\item Let be $h_1^\star$ as in Lemma \ref{LeJProp}.
\begin{enumerate}
\item It follows immediately from the definitions of $J_{h_0}$ and $J_\infty$.
\item Since:
\begin{equation}\label{F1h0}
\frac{\partial F_{1,h_0}(x)}{\partial h_0}>0\quad\forall\,x>0,
\end{equation}
it follows that:
\begin{equation*}\label{Wh0}
\frac{\partial W_{h_0}(x)}{\partial h_0}<0\quad\forall\,x>0.
\end{equation*}
Then, as we also know from Lemma \ref{LeFGProp} that $F_2$ is an increasing function, we have that:
\begin{equation*}\label{F2Wh0}
\frac{\partial}{\partial h_0}\left(F_2\left(\sqrt{\alpha_{12}}W_{h_0}(x)\right)\right)<0\quad\forall\,x>0.
\end{equation*}
Therefore:
\begin{equation}\label{Fh0}
\frac{\partial F_{h_0}(x)}{\partial h_0}>0\quad\forall\,x>0.
\end{equation} 
We also have from (\ref{F1h0}) that:
\begin{equation}\label{Gh0}
\frac{\partial G_{h_0}(x)}{\partial h_0}<0\quad\forall\,x>0.
\end{equation} 
Then, it follows from (\ref{FPos}), (\ref{Fh0}), (\ref{Gh0}) and the Leibnitz rule that:
\begin{equation*}
\frac{\partial J_{h_0}(x)}{\partial h_0}>0\quad\forall\,x\in(0,\nu_{h_0}).
\end{equation*}
Therefore, $\left\{\nu_{h_0}\right\}_{h_0\geq h_1^\star}$ is an increasing sequence of numbers and (\ref{JMonot}) holds.
\end{enumerate}
\item It is a direct consequence of the previous item and the definitions of $\xi_{h_0}$ and $\xi_\infty$ as the only one solutions to the equations (\ref{EqXi}) and (\ref{EqXiStar}), respectively.
\end{enumerate}
\end{proof}

\begin{lemma}\label{LeOrden}
Let be $h_1^\star$ as in Lemma \ref{LeJProp}. Then, there exists a number $h_0^{\star\star}\geq h_1^\star$ such that the following estimations holds when $h_0\to\infty$:
\begin{equation}\label{JOrden}
J_{h_0}(x)-J_\infty(x)=\mathcal{O}\left(\frac{1}{h_0}\right)
\quad\forall\,x\in[\xi_{h_0^{\star\star}},\nu_\infty].
\end{equation}
Therefore:
\begin{subequations}
\begin{align}
\label{xiOrden}&\xi_{h_0}-\xi_\infty=\mathcal{O}\left(\frac{1}{h_0}\right)\\
\label{muOrden}&\mu_{h_0}-\mu_\infty=\mathcal{O}\left(\frac{1}{h_0}\right)
\end{align}
\end{subequations}
when $h_0\to\infty$.
\end{lemma}

\begin{proof}
Let be $x\in[\xi_{h_1^\star},\nu_\infty]$ and $h_0\geq h_1^\star$. We have from Lemma \ref{LeJConv} that:
\begin{equation}\label{JJInf-1}
0<J_\infty(x)-J_{h_0}(x)=
\frac{H_{h_0}(x)-H_\infty(x)}{H_\infty(x)H_{h_0}(x)}
+\frac{\theta_0\sqrt{k_2c_2}}{D_\infty\sqrt{k_1c_1}}
\left(\frac{F_2\left(\sqrt{\alpha_{12}}W_{h_0}(x)\right)}{G_{h_0}(x)}-\frac{F_2\left(\sqrt{\alpha_{12}}W_\infty(x)\right)}{G_\infty(x)}\right).
\end{equation}
On one hand, we know from \cite{Ta2015-b} that there exist a function $\mathcal{J}_1$ and a positive number $h_0^{\star\star}\geq h_1^\star$ such that:
\begin{equation}\label{HHInf-1}
0<H_{h_0}(x)-H_\infty(x)\leq \frac{\mathcal{J}_1(x)}{h_0},\quad\forall\,h_0\geq h_0^{\star\star}.
\end{equation}
Then, since $\left\{H_{h_0}\right\}_{h_0\geq h_0^{\star\star}}$ is a decreasing sequence of functions which punctually converges to $H_\infty$ when $h_0\to\infty$, it follows that:
\begin{equation}\label{HHInf-2}
0<\frac{H_{h_0}(x)-H_\infty(x)}{H_\infty(x)H_{h_0}(x)}<\frac{\mathcal{J}_2(x)}{h_0}\quad\forall\,h_0\geq h_0^{\star\star},
\end{equation}
where $\mathcal{J}_2$ is the function defined by:
\begin{equation}\label{JJ2}
\mathcal{J}_2(x)=\frac{\mathcal{J}_1(x)}{H_\infty^2(x)},\quad x>0.
\end{equation}
On the other hand, since $\left\{W_{h_0}\right\}_{h_0\geq h_0^{\star\star}}$ is a decreasing sequence of functions which converges to $W_\infty$ when $h_0\to\infty$ and $F_2$ is an increasing function, we have that:
\begin{equation}\label{F2-1}
0<F_2\left(\sqrt{\alpha_{12}}W_{h_0}(x)\right)-F_2\left(\sqrt{\alpha_{12}}W_\infty(x)\right)\quad\forall\,h_0\geq h_0^{\star\star}.
\end{equation}
Then, as $\left\{G_{h_0}\right\}_{h_0\geq h_0^{\star\star}}$ is a decreasing sequence of functions which punctually converges to $G_\infty$ when $h_0\to\infty$, it follows that:
\begin{equation}\label{F2-2}
\begin{split}
0&<\frac{F_2\left(\sqrt{\alpha_{12}}W_{h_0}(x)\right)}{G_{h_0}(x)}-\frac{F_2\left(\sqrt{\alpha_{12}}W_\infty(x)\right)}{G_\infty(x)}\\
&<\frac{1}{G_\infty(x)}\left(F_2\left(\sqrt{\alpha_{12}}W_{h_0}(x)\right)-F_2\left(\sqrt{\alpha_{12}}W_\infty(x)\right)\right)\\
&<\frac{\mathcal{J}_3(x)}{h_0}\quad\forall\,h_0\geq h_0^{\star\star}.
\end{split}
\end{equation}
where $\mathcal{J}_3$ is the function defined by:
\begin{equation}\label{JJ3}
\mathcal{J}_3(x)=\frac{L_2\gamma k_1}{2D_\infty\sqrt{\alpha_2}}\frac{\exp(x^2)}{G_\infty(x)},\quad x>0,
\end{equation}
and $L_2$ is a Lipschitz constant for $F_2$ in $\left[W_\infty(\xi_{h_0^{\star\star}}),W_{h_0^{\star\star}}(\nu_\infty)\right]$.
Henceforth, we have from (\ref{JJInf-1}), (\ref{HHInf-2}) and (\ref{F2-2}) that:
\begin{equation}\label{JJInf-2}
0<J_\infty(x)-J_{h_0}(x)<\frac{\mathcal{J}(x)}{h_0}\quad\forall\,h_0\geq h_0^{\star\star},
\end{equation}
where $\mathcal{J}$ is the function defined by:
\begin{equation}\label{JJ}
\mathcal{J}(x)=\mathcal{J}_2(x)+\frac{\theta_0\sqrt{k_2c_2}}{D_\infty\sqrt{k_1c_1}}\mathcal{J}_3(x),\quad x>0.
\end{equation}
Therefore, (\ref{JOrden}) holds.

To prove (\ref{xiOrden}), we will use some geometric arguments. Let be $T$ the right triangle with vertices $P_1(\xi_{h_0},J_{h_0}(\xi_{h_0}))$, $P_2(\xi_{h_0},J_\infty(\xi_{h_0}))$, $P_3(\xi_\infty,J_\infty(\xi_\infty))$. Then, we have that:  
\begin{equation}\label{XiXiInf-1}
0<\xi_\infty-\xi_{h_0}=\frac{J_\infty(\xi_{h_0})-J_{h_0}(\xi_{h_0})}{\tan(\alpha_{h_0})},
\end{equation}
where $\alpha_{h_0}$ is the inner angle of $T$ with vertex $P_3$. Let also be $\tan(\widetilde{\alpha}_{h_0})$, $\widetilde{\alpha}_{h_0}\in(0,\pi)$, the slope of the secant line to the graph of $J_\infty$ which contains the points $P_2$ and $P_3$, and let be $\tan(\widetilde{\beta}$, $\beta\in(0,\pi)$, the slope of the tangent line at $P_3$ of the same graph. Since $\xi_{h_0}<\xi_\infty$ and $J_\infty$ is a decreasing convex function in $[\xi_{h_0^{\star\star}},\nu_\infty]$, we have that:
\begin{equation*}
\widetilde{\alpha}_{h_0}<\beta
\hspace{1cm}\text{and}\hspace{1cm}
\widetilde{\alpha}_{h_0},\,\beta\in\left(\frac{\pi}{2},\pi\right).
\end{equation*}
Then:
\begin{equation}\label{AlphaBeta}
\tan(\alpha_{h_0})>\tan(-\beta)=-J'_\infty(\xi_\infty)>0,
\end{equation}
since $\alpha_{h_0}=\pi-\widetilde{\alpha}_{h_0}$. Therefore, it follows from (\ref{JJInf-2}), (\ref{XiXiInf-1}) and (\ref{AlphaBeta}) that:
\begin{equation}\label{XiXiInf-2}
0<\xi_\infty-\xi_{h_0}<\frac{\mathcal{J}(\xi_{h_0})}{-J_\infty'(\xi_\infty)}\frac{1}{h_0}\quad\forall\,h_0\geq h_0^{\star\star}.
\end{equation}
We know from \cite{Ta2015-c} that $\mathcal{J}_1$ can be considered as given by:
\begin{equation*}
\mathcal{J}_1(x)=\frac{k}{\sqrt{\pi\alpha_1}}\frac{\exp(-x^2)}{\erf^2(x)}\left(x+\gamma(1-\epsilon)\frac{\sqrt{\pi}}{D_\infty}\frac{1}{F_{1,h_0^\star}(x)}\right)\frac{1}{F_{1,\infty}(x)F_{1,h_0^\star}(x)}.
\end{equation*}
Then:
\begin{equation}\label{JJ2Cota}
\begin{split}
\mathcal{J}_2(\xi_{h_0})
&=\frac{F_{1,\infty}(\xi_{h_0})}{G^2_\infty(\xi_{h_0})}\frac{k}{\sqrt{\pi\alpha_1}}\frac{\exp(-\xi_{h_0}^2)}{\erf^2(\xi_{h_0})}\left(\xi_{h_0}+\gamma(1-\epsilon)
\frac{\sqrt{\pi}}{D_\infty}\frac{1}{F_{1,h_0^\star}(\xi_{h_0})}\right)\frac{1}{F_{1,h_0^\star}(\xi_{h_0})}<\mathcal{M}_1,
\end{split}
\end{equation}
where $\mathcal{M}_1$ is the number defined by:
\begin{equation*}\label{MM1}
\mathcal{M}_1=\frac{k}{\sqrt{\pi\alpha_1}}\frac{F_{1,\infty}(\xi_{h_0^{\star\star}})}{G^2_\infty(\xi_{h_0^{\star\star}})F_{1,h_0^\star}(\nu_\infty)\erf^2(\xi_{h_0^{\star\star}})}\left(\nu_\infty+
\frac{\gamma(1-\epsilon)\sqrt{\pi}}{D_\infty}\frac{1}{F_{1,h_0^\star}(\nu_\infty)}\right)>0.
\end{equation*}
We also have that:
\begin{equation}\label{JJ3Cota}
\frac{\theta_0\sqrt{k_2c_2}}{D_\infty\sqrt{k_1c_1}}\mathcal{J}_3(x)<\mathcal{M}_2,
\end{equation}
where $\mathcal{M}_2$ is the number defined by:
\begin{equation*}\label{MM2}
\mathcal{M}_2=\frac{\theta_0L\gamma k_1\sqrt{k_2c_2}}{2D^2_\infty\sqrt{k_1c_1\alpha_2}}\frac{\exp(\nu_\infty^2)}{G_\infty(\xi_{h_0^{\star\star}})}.
\end{equation*}
Then, it follows from (\ref{XiXiInf-2}), (\ref{JJ2Cota}) and (\ref{JJ3Cota})  that:
\begin{equation}
0<\xi_\infty-\xi_{h_0}<\frac{\mathcal{M}}{h_0}\quad\forall\,h_0\geq h_0^{\star\star},
\end{equation}
where $\mathcal{M}$ is the number defined by:
\begin{equation*}\label{MM}
\mathcal{M}=\frac{\mathcal{M}_1+\mathcal{M}_2}{-J_\infty'(\xi_\infty)}>0.
\end{equation*}
Then, (\ref{xiOrden}) holds. 

Finally, we have that:
\begin{equation*}
\begin{split}
\left|\mu_{h_0}-\mu_{\infty}\right|&\leq
\sqrt{\alpha_{12}}\left(\frac{\mathcal{M}}{h_0}+\frac{\gamma\sqrt{\pi}}{2D_\infty}\left(\exp(\xi_\infty^2)\erf(\xi_\infty)-\exp(\xi_{h_0}^2)\erf(\xi_{h_0})\right)+\frac{\gamma k_1\exp(\nu_\infty^2)}{2D_\infty\sqrt{\alpha_1}}\frac{1}{h_0}\right)\\
&\leq\frac{\mathcal{M}_3}{h_0}\quad\forall\,h_0\geq h_0^{\star\star},
\end{split}
\end{equation*}
where $\mathcal{M}_3$ is the number defined by:
\begin{equation*}
\mathcal{M}_3=\sqrt{\alpha_{12}}\left(\mathcal{M}\left(1+\frac{\gamma\sqrt{\pi}L_6}{2D_\infty}\right)+\frac{\gamma k_1\exp(\nu_\infty^2)}{2D_\infty\sqrt{\alpha_1}}\right)>0
\end{equation*}
and $L_6$ is a Lipschitz constant in $\left[\xi_{h_0^{\star\star}},\nu_\infty\right]$ for the function $F_6$ defined by:
\begin{equation*}
F_6(x)=\exp(x^2)\erf(x),\quad x>0.
\end{equation*}
\end{proof}

We are now in a position to prove Theorem \ref{ThOrden}:

\begin{proof}
(of Theorem \ref{ThOrden})\\
Let be $x>0$ and $t>0$. We have that:
\begin{equation*}
\begin{split}
\left|\theta_{1,h_0}(x,t)-\theta_{1,\infty}(x,t)\right|&\leq
\frac{D_\infty}{1+\frac{\sqrt{\alpha_1\pi}}{k_1}\erf(\xi_{h_0^{\star\star}})}\frac{1}{h_0}\left[1+\frac{1}{\erf(\xi_{h_0})}\left(\frac{h_0\sqrt{\alpha_1\pi}}{k_1}\left(\erf(\xi_\infty)-\erf(\xi_{h_0})\right)+1\right)\right]\\
&\leq\frac{\mathcal{M}_{\theta_1}}{h_0}\quad\forall\,h_0\geq h_0^{\star\star},
\end{split}
\end{equation*}
where $\mathcal{M}_{\theta_1}$ is the number defined by:
\begin{equation*}
\mathcal{M}_{\theta_1}=\frac{D_\infty}{1+\frac{\sqrt{\alpha_1\pi}}{k_1}\erf(\xi_{h_0^{\star\star}})}\left[1+\frac{1}{\erf(\xi_\infty)}\left(\frac{L\sqrt{\alpha_1\pi}}{k_1}\mathcal{M}+1\right)\right]>0
\end{equation*}
and $L$ is a Lipschitz constant for the error function. Then (\ref{theta1Orden}) holds.

We also have that:
\begin{equation*}
\begin{split}
\left|\theta_{2,h_0}(x,t)-\theta_{2,\infty}(x,t)\right|\leq
\frac{2\theta_0}{\erfc^2(\mu_\infty)}\left(\erf(\mu_\infty)-\erf(\mu_{h_0})\right)\leq \frac{\mathcal{M}_{\theta_{2}}}{h_0}\quad\forall\,h_0\geq h_0^{\star\star},
\end{split}
\end{equation*}
where $\mathcal{M}_{\theta_2}$ is the number defined by:
\begin{equation*}
\mathcal{M}_{\theta_2}=\frac{2\theta_0L\mathcal{M}_3}{\erfc^2(\mu_\infty)}>0.
\end{equation*}
Therefore, (\ref{theta2Orden}) also holds. 

The proofs of (\ref{sOrden}), (\ref{rOrden}) follow straightforward from (\ref{xiOrden}) and (\ref{muOrden}).
\end{proof}

\subsection*{Conclusions}
In this article we have considered a two-phase solidification process for a one-dimensional semi-infinite material. We have assumed that the phase-change process starts from a constant bulk temperature imposed in the vicinity of the boundary and we have modelled it through a convective condition. Regarding the interface between solid and liquid phases, we have assumed the existence of a mushy zone and we have represented it by following the model of Solomon, Wilson and Alexiades. For this problem we have obtained a similarity solution that depends on a dimensionless parameter, which is defined as the only one solution to a transcendental equation. Moreover, we have analysed the relationship between the problems with convective and temperature boundary conditions and we have established when both problems are equivalent. We have also proved that the solution to the problem with the temperature boundary condition can be obtained from the solution to a problem with a convective boundary condition when the heat transfer coefficient at the fixed boundary goes to infinity and we have given the order of that convergence.

\subsection*{Acknowledgements}
This paper has been partially sponsored by the Project PIP No. 0534 from CONICET-UA (Rosario, Argentina) and AFOSR-SOARD Grant FA 9550-14-1-0122.

%-------------------------------------------------------------------------------------
%\nocite{*}
\bibliographystyle{plain}
\bibliography{References_2016-08-25}
%-------------------------------------------------------------------------------------

\end{document}